\documentclass[a4paper,11pt,final]{amsart}
\usepackage[margin=7cc]{geometry}
\usepackage[textsize=footnotesize,
color=yellow!50]{todonotes}

\usepackage[utf8]{inputenc}
\usepackage[T1]{fontenc}
\usepackage[english]{babel}
\usepackage{csquotes}
\usepackage{amssymb}
\usepackage{amsmath,amsthm}
\usepackage{showkeys}
\usepackage{datetime}

\usepackage{draftwatermark}
\SetWatermarkText{\today\  (\xxivtime)}
\SetWatermarkScale{0.4}

\usepackage[shortlabels]{enumitem}
\usepackage[pdfauthor={V. Fischer, D. Schrittesser, T. Weinert},
            pdftitle={Definable mad families and forcing axioms},
            pdfproducer={Latex with hyperref}]{hyperref}

\theoremstyle{plain}
\newtheorem{theorem}{Theorem}[section]
\newtheorem{lemma}[theorem]{Lemma}

\newtheorem{corollary}[theorem]{Corollary}
\newtheorem{claim}[theorem]{Claim}
\newtheorem{question}[theorem]{Question}

\newtheorem*{claim*}{Claim}
\newtheorem*{theorem*}{Theorem}
\newtheorem*{lemma*}{Lemma}
\newtheorem*{proposition*}{Proposition}
\newtheorem*{corollary*}{Corollary}
\newtheorem*{subclaim*}{Subclaim}
\newtheorem{convention}[theorem]{Notation}

\theoremstyle{definition}
\newtheorem{definition}[theorem]{Definition}

\newtheorem{fact}[theorem]{Fact}

\newtheorem*{definition*}{Definition}
\newtheorem*{example*}{Example}
\newtheorem*{fact*}{Fact} 

\theoremstyle{remark}
\newtheorem{remark}[theorem]{Remark}

\newenvironment{teqpar}[1][]{\begin{equation}\tag*{#1}\begin{minipage}{0.8\columnwidth}}%
{\end{minipage}\end{equation}}

\newenvironment{eqpar}{\begin{equation}\begin{minipage}{0.8\columnwidth}}%
{\end{minipage}\end{equation}}

\newenvironment{eqpar*}{\begin{equation*}\begin{minipage}{0.8\columnwidth}}%
{\end{minipage}\end{equation*}}

\newcommand{\BFA}{\mathsf{BFA}}
\newcommand{\BPFA}{\mathsf{BPFA}}
\newcommand{\BAAFA}{\mathsf{BA\mspace{-5.7mu}AFA}}
\newcommand{\MA}{\mathsf{MA}}

\DeclareMathOperator{\code}{\mathsf{cd}}
\DeclareMathOperator{\decode}{\mathsf{dc}}

\DeclareMathOperator{\powerset}{\mathcal{P}}
\newcommand{\ON}{\mathrm{On}}

\DeclareMathOperator{\nat}{\mathbb{N}}

\DeclareMathOperator{\PP}{\mathbb{P}}

\providecommand{\int}{\mathbb{Z}}
\providecommand{\reals}{\mathbb{R}}
\newcommand{\wlg}{w.l.o.g.}

\newcommand{\la}{\langle}
\newcommand{\ra}{\rangle}

\DeclareMathOperator{\ran}{ran}

\DeclareMathOperator{\seqop}{Seq}
\newcommand{\seqdc}[1]{\seqop( #1 )}

\providecommand{\res}{\mathbin{\upharpoonright} }

\newcommand{\ZFminus}{\ensuremath{\mathrm{ZF}^-}}

\providecommand{\ZFC}{\mathsf{ZFC}}
\providecommand{\ZF}{\textup{ZF}}

\providecommand{\Hhier}{\mathbf{H}}

\DeclareMathOperator{\eL}{\mathbf{L}}
\DeclareMathOperator{\Ve}{\mathbf{V}}

\providecommand{\setdef}{\;|\;}

\newcommand{\opair}[2]{
\mathord \langle #1, #2\rangle
}
\newcommand{\seq}[2]{\langle #1 | #2\rangle}
\newcommand{\soast}[2]{\{ #1 | #2\}}
\newcommand{\believes}[2]{#1 \vDash ``#2"}
\newcommand{\ifff}{\text{ if and only if }}
\DeclareMathOperator{\comp}{\parallel}
\newcommand{\dfeq}{:=}

\DeclareMathOperator{\On}{On}

\providecommand{\card}[1]{\lvert#1\rvert}

\providecommand{\ccc}{\textup{ccc}}

\providecommand{\Coll}{\textup{Coll}}

\author[Fischer]{Vera Fischer}
\address{Kurt Gödel Research Centre, Institut für Mathematik, UZA 1, Universität Wien. Augasse 2--6, 1090 Vienna, Austria}
\email{vera.fischer@univie.ac.at}

\author[Schrittesser]{David Schrittesser}
\address{Kurt Gödel Research Centre, Institut für Mathematik, UZA 1, Universität Wien. Augasse 2--6, 1090 Vienna, Austria}
\email{david@logic.univie.ac.at}

\author[Weinert]{Thilo Weinert}
\address{Kurt Gödel Research Centre, Institut für Mathematik, UZA 1, Universität Wien. Augasse 2--6, 1090 Vienna, Austria}
\email{thilo.weinert@univie.ac.at}

\title[Definable mad families and forcing axioms]{Definable mad families and forcing axioms}

\subjclass[2010]{}


\begin{document}

\newcommand{\linspan}[1]{\hull( #1 )}

\begin{abstract}
 We show that under the Bounded Proper Forcing Axiom and an anti-large cardinal assumption, 
 there is a $\mathbf{\Pi}^1_2$ MAD family. 
\end{abstract}

\maketitle

\section{Introduction}

\paragraph{A}
In his famous article \cite{miller} A.\ Miller constucted several co-analytic infinite combinatorial objects, working under the Axiom of Constructibility, i.e., assuming $\Ve=\eL$.
One of the objects he constructed was a co-analytic MAD family.

\medskip

A MAD family is a collection $\mathcal A$ of infinite subsets of $\omega$ with the following two properties: 
Firstly, $\mathcal A$ is an almost disjoint (short: a.d.) family, that is, any two distince $a,a' \in \mathcal A$ are almost disjoint, i.e., $a \cap a'$ is finite; and moreover, for any infinite set $b\subseteq \omega$ there is $a\in\mathcal A$.
Secondly, $\mathcal A$ is \emph{maximal} among a.d. families under inclusion. 

As is well-known, Mathias proved that no infinite MAD family can be analytic. Miller's result showed that this is optimal, since it shows it to be consistent with $\ZFC$ that there is a co-analytic infinite MAD family.

\medskip


The definability of MAD families has been investigated under many natural extensions of the axiomatic system $\ZFC$. 
For example, as was shown relatively recently, if the Axiom of Determinacy holds in $\eL(\reals)$, no infinite  MAD family can be an element of $\eL(\reals)$; and
under the Axiom of Projective determinacy, there is no projective infinite MAD family \cite{neeman-norwood,karen,pnas}.

Another natural family of extensions of $\ZFC$ are \emph{forcing axioms}.
It is clear from the above that forcing axioms essentially rule out the existence of definable infinite MAD families, provided such an axiom is strong enough to imply that the Axiom of Determinacy holds in $\eL(\reals)$.

\medskip

On the other hand, as we show in this article, under an \emph{anti-large cardinal assumption}, forcing axioms can lead to the opposite result: They imply the existence of projective infinite MAD families. Denote by $\BPFA$ the Bounded Proper Forcing Axiom.
\begin{theorem}\label{t.main}
Suppose $\BPFA$ holds and that $\omega_1$ is not remarkable in $\eL$.
Then there is a $\mathbf{\Pi}^1_2$ MAD family.
\end{theorem}
We take this  as evidence that under certain forcing axioms and anti-large cardinal assumptions, the universe behaves somewhat like $\eL$. This idea is also corroborated by the proof of the above theorem.

We can also view this theorem as result regarding the consistency strength of a certain theory:
\begin{corollary}\label{c.main}
The theory $\ZFC + \BPFA +$ ``there is no $\mathbf{\Pi}^1_2$ MAD familiy'' has consistency strength of at least a remarkable cardinal.
\end{corollary}
This is indeed remarkable, since $\ZFC + \BPFA$ is known to have consistency strength of a $\Sigma_2$-reflecting cardinal, which is weaker than a remarkable cardinal.

\medskip

\paragraph{B}
Our work has some precursors in the literature: 
In \cite{caicedo-friedman} it is shown that under $\BPFA$, if $\omega_1$ is not remarkable in $\eL$ every predicate on $\powerset(\omega)$ which has a $\mathbf{\Sigma}_1$ definition in $\Hhier(\omega_2)$ also has a $\mathbf{\Sigma}^1_3$ definition. 

It was shown by Asger Törnquist in \cite{asger} that if there is an infinite  $\mathbf{\Sigma}^1_2$ MAD family, there is an infinite $\mathbf{\Pi}^1_1$ MAD family.
Unfortunately, the latter proof does not lift to show that there exists a $\mathbf{\Pi}^1_2$ infinite MAD under $\BFA + \omega_1$ is not remarkable in $\eL$.
The reason for this is that T\"ornquist's relies on properties of $\Sigma^1_2$ and $\Pi^1_1$ sets which do not hold for $\Sigma^1_3$ and $\Pi^1_2$ sets. 

Moreover, our proof can easily be adapted (we leave this to the reader) to show, e.g., that there is a $\Pi^1_2$ maximal eventually different family. For this type of family, no analogue of T\"ornquist's theorem is known.

\medskip

\paragraph{C}
The paper is organized as follows. In section \S2 we discuss a result of Caicedo and Velickovic which can be summed up as follows: $\BPFA$ implies that there is a well-ordering of $\powerset(\omega)$ of length $\omega_2$ with definable initial segments.
In \S 3 we discuss the role of the anti-large cardinal assumption, referring to work of Schindler, and discuss a technique of localisation which we have used before (e.g., \cite{}) and which takes a particularly simple form under $\BFA + \omega_1$ is not remarkable in $\eL$.
Finally, in \S 4 we prove Theorem~\ref{t.main}.
We close with open questions in \S 5.

\medskip

{\it Acknowledgements:
The first, second and third authors would like to thank the Austrian Science Fund (FWF) for the generous support
through START Grant Y1012-N35. The second author would also like to thank the FWF through generous support from project P 29999. 
}

\section{A well-ordering with $(\Sigma_1,\powerset(\omega_1))$-definable initial segments}

It was shown by Moore  that under $\BPFA$ there is a well-ordering of $\powerset(\omega)$ of order-type $\omega_2$.
Improving Moore's argument, Caicedo and Velickovic  \cite{caicedo-velickovic} obtained, under $\BPFA$, such a well-ordering which is definable by a $\Sigma_1$ formula with a parameter from $\powerset(\omega_1)$. 

In fact, their well-ordering has the following property which will be crucial to our argument.

\begin{theorem}\label{t.wo}
Under $\BPFA$ there is a well-ordering $\prec$ of $\powerset(\omega)$ such that for  some $\Sigma_1$ formula $\Phi_{\prec}(u,v,w)$  and some parameter $c_{\prec}\subseteq\omega_1$, 
\[
\big(\forall x \in \powerset(\omega)\big)(\forall I)\; \Big(\Phi_{\prec}(x,I,c_{\prec}) \iff I=\{y\in\powerset(\omega) : y \prec x\}\Big)
\]
\end{theorem}

For convenience, we give a name to this type of well-order:
\begin{definition}\label{d.good.wo}
We say a well-order $\prec$ of $\powerset(\omega)$ with the property from Theorem~\ref{t.wo} has 
\emph{$(\Sigma_1,\powerset(\omega_1))$-definable good initial segments}. 
\end{definition}
Such a well-ordering is obviously very useful when one is interested in devising a recursive definition of optimal complexity.
Note that \cite{caicedo-velickovic} shows that their well-order of $\powerset(\omega)$ has the following property. 
Of course, this is equivalent to having $(\Sigma_1,\powerset(\omega_1))$-definable initial segments.
\begin{fact}\label{f.wo}
That a well-ordering of $\powerset(\omega)$ has $(\Sigma_1,\powerset(\omega_1))$-definable initial segments (i.e., the property from Theorem \ref{t.wo}) is equivalent to the conjunction of the following:
\begin{itemize}
\item $\prec$ is $\Sigma_1$  with a parameter $c_{\prec}\subseteq\omega_1$,
\item There is a formula $\Phi_{\text{is}}(u)$ such that for any transitive model $M$ with $c_{\prec} \in M$,
$M \vDash \Phi_{\text{is}}(c_{\prec})$ if and only if $M\cap \powerset(\omega)$ is an initial segment of $\prec$.
\end{itemize}
\end{fact}
\begin{proof}
To see that the above implies that $\prec$ is well-ordering with $(\Sigma_1,\powerset(\omega_1))$-definable initial segments just
let $\Phi_\prec(x,I,c_\prec)$ be the formula
\begin{multline*}
(\exists M)\; \text{$M$ is a transitive $\in$-model,}
\text{$\{c_\prec,I\}\subseteq M$, and }\\
M\vDash\text{``}\Phi_{\text{is}}(c_{\prec}) \land I=\{y\in\powerset(\omega) \setdef y\prec x\}\text{''}
\end{multline*}
and observe $I=\{y\in\powerset(\omega) : y \prec x\} \iff \Phi_\prec(x,I,c_\prec)$.
For the other direction, simply let $\Phi_{\text{is}}(c_{\prec})$ be the formula
\[
\big(\forall x \in \powerset(\omega)\big)(\exists I)\; \Phi_{\prec}(x,I,c_{\prec}).\qedhere
\]
\end{proof}

For a proof that under $\BPFA$ there is such a well-ordering of $\powerset(\omega)$ with $(\Sigma_1,\powerset(\omega_1))$-definable good initial segments, we refer the reader to the excellent exposition in  \cite{caicedo-velickovic}. 

\section{Coding, reshaping, and localization}

We start by recalling the following well-known fact. Let $\mathcal B=\la b_\xi : \xi < \omega_1\ra$ be an arbitrary sequence of pairwise almost disjoint infinite subsets of $\omega$.
\begin{fact}\label{f.adc}
Under $\MA_{\aleph_1}$, for every subset of $S\subseteq \omega_1$ there is a $c\subseteq \omega$ such that
\begin{equation}\label{e.adc}
S=\{\xi<\omega_1 : c \cap b_\xi\text{ is infinite}\}.
\end{equation}
\end{fact}
The proof of this fact is equally well-known; it uses Solovay's almost disjoint coding (see, e.g., \cite{kunen}).

We take the opportunity to introduce the following rather natural terminology:
\begin{definition} Let $\mathcal B=\la b_\xi : \xi < \omega_1\ra$ be an arbitrary sequence of pairwise almost disjoint infinite subsets of $\omega$. We shall say that $c \subseteq \omega$ \emph{almost disjointly via $\mathcal B$ codes} the set $S$ to mean
precisely that \eqref{e.adc} holds.
\end{definition}

\medskip

Our only use of the assumption that $\omega_1$ is not remarkable in $\eL$ is in the following fact (this was shown by Ralf Schindler in \cite{schindler}).
\begin{fact}\label{l.remarkable}
Suppose $\omega_1$ is not remarkable in $\eL$ and $\BPFA$ holds. Then there exists $r\in \powerset(\omega)$ such that $\omega_1=(\omega_1)^{\eL[r]}$.
\end{fact}

\begin{convention}~\label{n.r}
\begin{enumerate}
\item For the rest of this article, let us suppose that $\omega_1=(\omega_1)^{\eL[r]}$ for some $r\in\powerset(\omega)$ which from now on shall remain fixed.
\item Fix an almost disjoint family $\mathcal F=\la f_\xi : \xi <\omega_1\ra$ which has a $\Sigma_1$ definition in $\eL[r]$ and such that
for any $\alpha < \omega_1$, $\la f_\xi : \xi < (\omega_1)^{\eL_\alpha[r]}\ra$ is the set satisfying this definition in  $\eL_\alpha[r]$.
\end{enumerate}
\end{convention}

It is a consequence of $\omega_1=(\omega_1)^{\eL[r]}$ and $\MA_{\aleph_1}$ that any predicate which is $\Sigma_1$ in $\Hhier(\omega_2)$ (with a parameter) can be localized in a strong sense. 
A version of this result 
can, e.g., be found in \cite{caicedo-friedman}. Said paper \cite{caicedo-friedman} also served as an important motivation for the present article.

To state the following localization lemma, let us make a definition which will be used throughout the paper.
\begin{definition}[Suitable models]
A \emph{suitable model} is a countable transitive $\in$-model $N$ such that $r \in N$,
$N\vDash\ZF^-$ and ``$\omega_1$ exists''.
\end{definition}

\begin{lemma}[A form of localization]\label{l.witnessing}
Suppose $\MA_{\aleph_1}$ holds (and recall that we are working under the assumption that $\omega_1 = (\omega_1)^{\eL[r]}$ made in~\ref{n.r}).
Let $\phi(y,\omega_1)$ be an arbitrary formula 
formula, where $y \in \powerset(\omega)$ and $\omega_1$ are parameters, and suppose that for some transitive $\in$-model $M$ with $\{\omega_1,y\} \in M$ it holds that $M \vDash \phi(y,\omega_1)$. 
Then there is $c\subseteq \omega$ such that the following holds: \begin{eqpar}\label{e.c}
Given any suitable model $N$ with $\{c,y\}\subseteq N$ the following must hold in $N$: ``There is a transitive $\in$-model $M^*$ such that $\{y,(\omega_1)^N\}\subseteq M^*$ and $M^*\vDash\phi(y,(\omega_1)^N)$''.
\end{eqpar}
\end{lemma}
\begin{proof}
Fix a transitive model $M$ as in the lemma. 
Find $S \subseteq \omega_1$ such that via G\"odel pairing, $S$ gives rise to a well-founded binary relation $S^*$ on $\omega_1$ whose transitive collapse is $\la M, \epsilon\res M\rangle$.
We can ask that $y$ and $\omega_1$ are mapped to specific points of in $\la,\omega_1,S^*\ra$ by the inverse of the collapsing map, say to $0$ and $1$.

Let 
\[
D =\{\beta\in\omega_1 : (\exists \mathcal{N^*})\; \mathcal N^* \prec L_{\omega_2}[S^*,y], \{S^*,y\}\in \mathcal N^*, \beta=\omega_{1}\cap \mathcal N^*\}.
\]
For $Y\subseteq\ON$, let $\mathrm{Even}(Y)=\{\xi : 2\xi \in Y\}$ and $\mathrm{Odd}(Y)=\{\xi : 2\xi+1 \in Y\}$. 
Choose $ Y$ to be any subset of $\omega_1$ such that $\mathrm{Even}(Y)=S^*$ and for each $\beta \in D$, the preimage under $G$ of $\mathrm{Odd}(Y\cap [\beta, \beta+\omega)$ is a well-founded binary relation of rank at least $\min\big(D\setminus (\beta+1)\big)$.

\begin{claim}\label{c.witnessing}
$Y\subseteq \omega_1$ satisfies the following:
\begin{eqpar}\label{e.Y}
Given any suitable model $N$ with $\{Y \cap (\omega_1)^N,y\}\subseteq N$ the following must hold in $N$: ``There is a transitive $\in$-model $M^*$ such that $\{y,(\omega_1)^N\}\subseteq M^*$ and $M^*\vDash\phi(y,(\omega_1)^N)$''.
\end{eqpar}
 \end{claim}
 \begin{proof}
 Too see that $Y$ indeed satisfies \eqref{e.Y} let $\mathcal N$ as in \eqref{e.Y} be given.
 By choice of $Y$, $\beta=(\omega_{1})^{N} \in D$. 
 Thus there a transitive model $\overline{N}^*$ and an elementary embedding $j: \overline{N}^* \to L_{\omega_2}[S^*,y]$ with critical point $\beta$ and such that $\{S^*,y,\omega_1\} \subseteq \ran(j)$. 
By elementarity 
 $\overline{N}^*\vDash$``The transitive collapse of $\la \beta, S^*\cap \beta\ra$ is a transitive $\in$-model $M^*$ of $\Phi(y,\beta)$''.
But taking the transitive collapse is absolute, so  $N$ must satisfy the same sentence.
\renewcommand{\qedsymbol}{{\tiny Claim \ref{c.witnessing}.} $\Box$}
\end{proof} 
Finally, we find $c\in\powerset(\omega)$ which almost disjointly via $\mathcal F$ codes the set $Y\subseteq\omega_1$ constructed above.
By a proof identical to that of Fact~\ref{l.remarkable}, the real $c$ satisfies \eqref{e.c}, proving the lemma. 
\renewcommand{\qedsymbol}{{\tiny Lemma \ref{l.witnessing}.} $\Box$}
\end{proof}

\section{Proof of Theorem \ref{t.main}}

In this section we prove Theorem~\ref{t.main} in the following, slightly more general form:
\begin{theorem}\label{t.better}
Suppose $\powerset(\omega)$ has a well-ordering of length $\omega_2$ with $(\Sigma_1,\powerset(\omega_1))$-definable initial segments, $\MA_{\aleph_1}$ holds, and  $\omega_1 = (\omega_1)^{\eL[r]}$ for some $r\in\powerset(\omega)$.
Then there is a $\mathbf{\Pi}^1_2$ MAD family.
\end{theorem}
It is clear by Theorem~\ref{t.wo} and Lemma~\ref{l.remarkable} that $\BPFA$ implies the hypothesis, so proving the above theorem will indeed prove Theorem~\ref{t.main}.

\medskip

We supress the parameter $r$ and assume $\omega_1 = (\omega_1)^{\eL}$; our argument will relativize to $r$ trivially.
\begin{convention}~\label{n.wo}
By Theorem~\ref{t.wo} we can fix a well-ordering $\prec$ of $\powerset(\omega)$ with $(\Sigma_1,\powerset(\omega_1))$-definable initial segments, together with a parameter $c_{\prec}\subseteq\omega$ and a formula $\Phi_{\textup{is}}(c_{\prec})$ as in Fact~\ref{f.wo}. 
\end{convention}

We shall inductively construct a sequence $\la a_\nu : \nu < \omega_2\ra$ such that $\mathcal A = \{a_\nu : \nu < \omega_2\}$ will be a $\mathbf{\Pi}^1_2$ MAD family.

The most straightforward formula defining a MAD family $\mathcal A$ would express that $a\in \mathcal A$ iff there is an initial segment $\la a_\nu : \nu \leq  \xi\ra$ of the construction with $a=a_\xi$; that is, assuming we can find a formula expressing that $\la a_\nu : \nu \leq  \xi\ra$ is an initial segment of this construction.
But of course it is not clear how any projective formula should express such a fact about $\la a_\nu : \nu < \xi\ra$, this being an object of size $\omega_1$. 
A first step towards a solution is that $a_\xi$ should \emph{code} certain sets of size $\omega_1$, including $\la a_\nu : \nu < \xi\ra$. 
Almost disjoint coding via $\mathcal F$ (see Fact~\ref{f.adc}) allows us to find a real coding these large sets, and then some reals `localizing' this coding, i.e., ensuring that coding an initial segment of the construction is expressible by a $\Pi^1_2$ formula.
Using a variant of the following coding from \cite{miller} we can then code these reals into $a_\xi$.

\subsection{Coding into an almost disjoint family}

We call the following fact from Miller's article \cite{miller} to the reader's attention.

\begin{fact}[see {\cite[Lemma~8.24,~p.~195]{miller}}]\label{f.coding.miller.simple}
Fix $z \in \powerset(\omega)$ and suppose $\vec a=\la a^\nu : \nu < \xi \ra$ is a countable sequence of pairwise almost disjoint infinite sets.
For any $d \in [\omega]^\omega$ which is almost disjoint from every element of $\vec a$ there is $a \in [\omega]^\omega$ such that
\begin{itemize}
\item $a \cap d$ is infinite, 
\item $a$ is disjoint from each $a^\nu$ for $\nu<\xi$, 
\item and $z$ is computable from $a$ and $\vec a \res \omega=\la a^n : n < \omega \ra$.
\end{itemize}
\end{fact}
Using this fact, Miller succeeds in constructing a co-analytic MAD family in $\eL$:
he recursively constructs $\la a_\nu : \nu < \omega_1\ra$ such that in the end, $\mathcal A = \{a_\nu : \nu < \omega_1\}$ turns out to be a 
$\mathbf{\Pi}^1_1$ MAD family.
At some initial  stage $\xi < \omega_1$ having constructed $\vec a=\la a^\nu : \nu < \xi \ra$ he considers a counterexample $d$ to the maximality of the family 
$\{a^\nu : \nu < \xi\}$
 constructed so far.
Instead of adding this set $d$ to $\vec a$, he adds $a$ as in the fact above, which in addition codes some information $z$ so as to bring down the definitional complexity of $\mathcal A$.

\medskip

Since we shall need a variant of this type of coding, let us repeat Miller's proof of the above fact.
\begin{proof}[Proof of Fact~\ref{f.coding.miller.simple}]
Let $\vec b=\langle b^n : n \in \nat\rangle$ enumerate $\{a^\nu :  \omega \leq \nu < \xi\}$.
For each $n\in\omega$, chose a finite set $G_n \subseteq a^n \setminus \bigcup \{ b^k : k< n\}$ so that
$\lvert G_n \rvert$ is even if $n\in z$, and odd otherwise.
Finally, let $a = d \cup \bigcup \{F_n : n\in \omega\}$.
\end{proof}

For our purposes the previous fact is useless, since as $2^\omega=\omega_2$ under $\BPFA$ we shall have to deal with uncountable sequences $\vec a=\la a^\nu : \nu < \xi \ra$. 
Interestingly, there is a variant of the above construction that allows us to deal with uncountable sequences.
\medskip

Before we describe this variant let us commit, once and for all, to some sequence (to be used for coding purposes) as an initial segment of the MAD family we are about to construct.
\begin{convention}\label{n.a}
Let us fix, for the rest of this article, some sequence $\vec a_\omega=\la a_n : n \in \omega\ra$ of infinite sets any two of which are almost disjoint.
\end{convention}

We now state our variant of Miller's coding lemma. 
For this variant, we must make an additional assumption (the existence of $c$ below) which will turn out to be innocent.

\begin{fact}\label{f.coding.miller}
Suppose $\vec a=\la a_\nu : \nu < \xi\ra$ is a (possibly uncountable) sequence of pairwise almost disjoint infinite subsets of $\omega$ such that $\vec a\res\omega=\vec a_\omega$. 
Further suppose we have $c \in [\omega]^\omega$ satisfying the following:
\begin{itemize}
\item $c$ is almost disjoint from each $a_\nu$, for $\omega\leq\nu<\xi$, and
\item $c \cap a_n$ is infinite for each $n\in\omega$.
\end{itemize}
Then for any $z \in \powerset(\omega)$
and any $d \in [\omega]^\omega$ which is almost disjoint from every element of $\vec a$ there is $a \in [\omega]^\omega$ such that
\begin{itemize}
\item $a \cap d$ is infinite, 
\item $a$ is disjoint from each $a_\nu$ for $\nu<\xi$, 
\item and $z$ is computable from $a$ and $\vec a \res \omega=\la a_n : n < \omega \ra$.
\end{itemize}
In fact there are functions $\decode \colon \powerset(\omega)  \to \powerset(\omega)$ and $\code \colon \powerset(\omega)^3 \to \powerset(\omega)$, both of which are computable in $\vec a_\omega$, such that 
$a$ as above is given by $a=\code(z,d,c)$ and $z$ can be recovered from $a$ as $z=\decode(a)$.
\end{fact}
The name $\decode$ was chosen to remind us that this function will be used to `decode' $z$ from $a$, and likewise, the name $\code$ should remind us that the function produces a `code' (for $z$).
\begin{proof}[Proof of Fact~\ref{f.coding.miller}]
We define $\code \colon \powerset(\omega)^3\to\powerset(\omega)$ as follows.
Let $F_n$ be the shortest finite initial segment of 
\[
c \cap a_n \setminus \big (d \cup \bigcup\{a_k : k< n\}\big)
\] 
such that $\lvert F_n \cup (d \cap a_n) \rvert$ is even if $n\in z$ and odd otherwise.
Clearly, $F_n$ can be found by a procedure which is computable in $\vec a_\omega$, $c$, $d$, and $z$.
Now define the function $\code$ by
\[
\code(d,c,z) = d \cup \bigcup \{F_n : n\in\omega\}.
\]
Moreover, we define $\decode \colon \powerset(\omega)  \to \powerset(\omega)$ as follows: 
Given $a \in [\omega]^\omega$ let 
\[
\decode(a)=\{n \in\omega : \lvert a \cap a_n\rvert\text{ is even}\}.
\]
Clearly, these functions satisfy the conditions in the lemma.
\end{proof}

\subsection{Minimal local witnesses} 

The functions $\code$ and $\decode$ together with the almost disjoint coding into reals of subsets of $\omega_1$ via $\mathcal F$ will help us arrange that $a_\xi$ codes $\la a_\nu : \nu<\xi\ra$.
But crucially, we need the fact that $a_\xi$ codes an initial segment of the construction (up to stage $\xi$, some ordinal below $\omega_2$) to be witnessed by a $\mathbf{\Pi}^1_2$ formula (the same formula for all $\xi<\omega_2$).
This involves uniquely selecting a real $c_\xi \in \powerset(\omega)$ which we call a \emph{minimal local witnesses} and whose task is to \emph{localize} the coding to suitable countable models. 
Uniquely selecting such a real is a non-trivial task, and to tackle it we introduce some terminology.

\begin{convention}\label{n.f}
For the remainder of this article, let $F\colon \omega^2\to \omega$ denote some fixed recursive bijection.
\end{convention}
\begin{definition}~\label{d.basic.coding}
\begin{enumerate}

\item Given $c \subseteq \omega$ and $n\in\omega$ we write $(c)_n$ for $\{m \in\omega : F(n,m)\in c\}$ (where $F$ is the bijection of $\omega$ with $\omega^2$ from~\ref{n.f} above).

\item Given $c \subseteq \omega$ we write $\seqdc{c}$ for the sequence
$\la (c)_n : n\in\omega\ra$.

\item We say \emph{$c \subseteq \omega$ almost disjointly via $\mathcal F$ codes a sequence $\vec b$  of length $< \omega_2$} to mean that $c$ almost disjointly via $\mathcal F$ codes $S\subseteq \omega_1$ such that interpreting $S$ as a binary relation $S^*$ on $\omega_1$ (via G\"odel pairing), this relation $S^*$ is well-founded, $S^*$ is isomorphic to $\in$ restricted to the transitive closure of $\vec b$, and that moreover $\vec b=\la b_\nu : \nu < \xi\ra$ is a sequence of length $\xi <\omega_2$.
\end{enumerate}
\end{definition}

\medskip

The crucial definition for our proof of Theorem~\ref{t.better} (and thus, of Theorem~\ref{t.main}) is that of \emph{minimal local witness}. 

\begin{remark}
In the end, our MAD family will be 
\[
\mathcal A = \{a\in [\omega]^\omega : c=\decode(a)\text{ is a minimal local witness and }a=\code\big((c)_0, (c)_1, c\big)\}
\]
We will show below that being a minimal local witness is expressible by a $\mathbf{\Pi}^1_2$ formula.
Thus, $\mathcal A$ will be $\mathbf{\Pi}^1_2$.
\end{remark}

Before we introduce the notion of minimal local witnesses, we make another convenient definition, for which some motivation should be provided by the previous remark.
\begin{definition}
We shall say that a sequence $\vec b=\la b_\nu : \nu<\xi\ra$ of length $\xi < \omega_2$ is a \emph{coherent candidate} if 
 $\vec a_\omega \subseteq \vec b$ and moreover, for each $\nu < \xi$ it holds that $b_\nu=\code\big((c_\nu)_0, (c_\nu)_1, c_\nu\big)$ where $c_\nu=\decode(b_\nu)$.
\end{definition}

We proceede towards the definition of minimal local witness, by defining the notions of $k$-witness, minimal $k$-witness  and $k$-localizer, by induction on $k \in \omega$, $k \geq 3$.

\begin{definition}
We say $\bar c \in \powerset(\omega)^3$ 
\emph{is a $3$-witness} if and only if
\begin{teqpar}[{$(*)_3$}]\label{e.0-order}
\begin{enumerate}[(a)]
\item $\bar c(2)$ almost disjointly  via $\mathcal F$ codes a sequence $\vec b=\la b_\nu : \nu<\xi\ra$.

\item $\vec b$ is a coherent candidate, i.e., $\vec a_\omega \subseteq \vec b$ and for each $\nu < \xi$, letting $c_\nu=\decode(b_\nu)$ it holds that $b_\nu=\code\big((c_\nu)_0, (c_\nu)_1, c_\nu\big)$.
 \item $\bar c(1)$ is subset of $\omega$ such that $c \cap b_\nu$ is infinite if $\nu<\omega$ and finite for all other $\nu<\xi$.

\item $\bar c(0)$ is an element of $[\omega]^\omega$ which is almost disjoint from each $b_\nu$ for $\nu<\xi$;
\end{enumerate} 
\end{teqpar}
\end{definition}

\begin{remark}
Clearly, the sequence $\vec b$ from (a) is intended to be an initial segment of the MAD family under construction.
We ask (b) as a step to ensuring that this is indeed the case.
The reader will notice that in (c) we require that $\bar c(1)$ has the same properties as $c$ in Fact~\ref{f.coding.miller},
and in (d) we require that $\bar c(0)$ has the same properties as $d$ in said fact. 
The reader may think of $\bar c(0)$ as a counterexample to maximality of $\vec b$ which we wish to eliminate at stage $\xi$ of our construction of $\mathcal A$ by adding a `self-coding' element to our MAD family which has infinite intersection with $\bar c(0)$. 
\end{remark}

\begin{lemma}\label{l.sigma-1}
That $\bar c$ is a $3$-witness can be expressed both by $\Sigma_1$ and a $\Pi_1$ formula, each with parameter $\omega_1$.
\end{lemma}
\begin{proof}
The stamtent in (a) that $\bar c(2)$ almost disjointly  via $\mathcal F$ codes a sequence $\vec b$ of length $<\omega_2$ is easily seen to be a $\Sigma_1$ property of $(\bar c,\vec b)$, allowing $\omega_1$ as a parameter; likewise the negation of this statement.
All the other statements (b)---(d) are obviously $\Delta_1$ in the parameters $\vec b$ and $\bar c$.
The lemma follows easily.
\end{proof}

We continue with the definition of minimal $3$-witness for a coherent candidate $\vec b$.

\begin{definition}
For any $3$-witness $\bar c \in \powerset(\omega)^3$, we say \emph{$\bar c$ is a witness to $\vec b$} with $\bar b$ the sequence coded by $\bar c(2)$ as in (a) above. We also write $\vec b(\bar c)$ for this sequence.
Write $\prec^3$ for the lexicographic ordering on $\powerset(\omega)^3$ induced by $\prec$;
we say $\bar c \in \powerset(\omega)^3$ is a \emph{minimal $3$-witness} if it is $\prec^3$-minimal among all $3$-witnesses to the same sequence $\vec b$.
\end{definition}

We now give the crucial definition of a \emph{localizer}---a real which ensures that witnesses can be recognized from a local property.

\begin{definition}
Given $\bar c  \in \powerset(\omega)^3$ (a putative $3$-witness) we say $c \in \powerset(\omega)$ is  
\emph{is a $3$-localizer for $\bar c$} if and only if the following holds:
\begin{teqpar}[{$(*)_4$}]\label{e.1-order}
For any suitable model $N$ with $\{\bar c, c, \vec a_\omega\}\subseteq N$, the following holds in $N$:
There is a transitive model $M$ of $\ZFminus$ such that  $M\vDash\Phi_{\text{is}}(c_{\prec})$, $\{\omega_1,\bar c,\vec a_\omega\}\subseteq M$, and 
\begin{enumerate}[(a)]
\item $M\vDash$``$\bar c$ is a minimal $3$-witness''.
\item Writing $\vec b(\bar c)$ as $\la b_\nu : \nu<\xi\ra$, it holds that $\vec b(\bar c) \in M$ and for each $\nu<\xi$, 
 $M\Vdash$``$\bar c_\nu\res 3$ is a minimal $3$-witness'', where
 $\bar c_\nu=\seqdc{\decode(b_\nu)}$.

\item $c \notin M$.
\end{enumerate} 
\end{teqpar}
\end{definition}

\begin{remark}
Note that ``$\bar c_\nu\res 3$ is a minimal $3$-witness'' is a statement which uses $\vec a_\omega$ as a parameter.
We ask Item (c) above because this will allow us to show that if $a$ has a minimal local witness and this witness codes $\vec b=\la b_\nu : \nu<\xi\ra$, then for each $\nu < \xi$ it must hold that $b_\nu \prec a$ (of course this must remain but a vague promise until we have given the complete definition of minimal local witness).
\end{remark}

We need the following crucial lemmas:
\begin{lemma}\label{l.global-local}
There exists a $3$-localizer $c$ for any minimal $3$-witness $\bar c \in \powerset(\omega)^3$.
\end{lemma}
\begin{proof}
Suppose $\bar c  \in \powerset(\omega)^3$ is a minimal $3$-witness.
Fix a model $M$ of $\ZF^-$ such that $\{\omega_1, \bar c\} \subseteq M$ and so that $M\vDash \Phi_{\text{is}}(c_{\prec})$.
Then statements such as ``$c$ codes almost disjointly via $\mathcal F$ the sequence $\vec b$ of length $<\omega_1$'' (cf.\ Definition~\ref{d.basic.coding} above) are absolute for $M$. 
So the property of being a $3$-witness is absolute for $M$, since it is $\Delta_1$ (allowing $\omega_1$ as a parameter). 
Since $\powerset(\omega)\cap M$ is an initial segment of $\prec$ (cf.\ Fact~\ref{f.wo} as well as~\ref{n.wo}), so is the property of being a minimal $3$-witness.

Now as in the proof of Lemma~\ref{l.witnessing}, find $c$ coding almost disjointly via $\mathcal F$ a subset of $\omega_1$ which is isomorphic to $\in\res M$ and such that for any suitable model $N$, if $c,\bar c \in N$ then it holds in $N$ that $c$ codes a model $M^*$ which witnesses the $\Sigma_1$ statement expressing $\bar c$ is a minimal $3$-witness. 
\end{proof}

\begin{lemma}\label{l.local-global}
Suppose $\bar c \in \powerset(\omega)^3$.
If there exists a $3$-localizer for $\bar c$, then $\bar c$ is a minimal $3$-witness.
\end{lemma}
\begin{proof}
Suppose $c$ is a $3$-localizer for $\bar c$. 
Let $\bar N$ be a countable elementary submodel of $\eL_{\omega_2}[c,\bar c, \vec a_\omega]$ with $\{\omega_1, c,\bar c,\vec a_\omega\}\subseteq N$ and let $N$ be the transitive collapse of $N$. 
Then $N$ is suitable, and so by \ref{e.1-order} the following holds in $N$:
There is a transitive model $M$ of $\ZFminus$ such that  $M\vDash\Phi_{\text{is}}(c_{\prec})$, $\{(\omega_1)^N,\bar c,\vec a_\omega\}\subseteq M$, and 
\begin{enumerate}[(a)]
\item $M\vDash$``$\bar c$ is a minimal $3$-witness''.
\item $\bar c$ codes $\vec b=\la b_\nu : \nu<\xi\ra$, $\vec b\in M$ and for each $\nu<\xi$, 
 $M\Vdash$``$\bar c_\nu\res 3$ is a minimal $3$-witness'', where
 $\bar c_\nu=\seqdc{\decode(b_\nu)}$.
\end{enumerate} 
By elementarity, there exists such a model $M$ in $\eL_{\omega_2}[c,\bar c, \vec a_\omega]$ with all of the above properties, where $(\omega_1)^N$ is replaced by $\omega_1$.
Since $M\vDash\Phi_{\text{is}}(c_{\prec})$ and so $\powerset(\omega)\cap M$ is an initial segment of $\prec$ and since being a $3$-witness is absolute for transitive models, also (a) above is absolute for $M$.
Hence $\bar c$ is a minimal $3$-witness, finishing the proof.

We point out that by (b) we also have that for each $\nu<\xi$, $\seqdc{\decode(b_\nu)}$ is a $3$-witness.
\end{proof}

Thus we have shown that $\bar c$ is a minimal $3$-witness if and only if there exists a $3$-localizer for $\bar c$.
Of course, there may be more than one $3$-localizer for a given $3$-minimal witness.

\begin{definition}
We say $\bar c \in \powerset(\omega)^4$ is a \emph{minimal $4$-witness} if and only if
$\bar c(3)$ is the $\prec$-least localizer for $\bar c\res 3$.
\end{definition}

We now continue the definition of minimal $k$-witness for by induction on $k$, following the template given by the definition for $k=4$, except that there is no longer any need to require (c). 

\begin{definition}
Suppose  
we have already defined what it means to be a minimal $k$-witness for elements of $\powerset(\omega)^{k}$.
Given $\bar c  \in \powerset(\omega)^{k}$ (a putative $k$-witness) we say $c \in \powerset(\omega)$ is  
\emph{is a $k$-localizer for $\bar c$} if and only if the following holds:
\begin{teqpar}[{$(*)_{k}$}]\label{e.k+1-order}
For any suitable model $N$ with $\{\bar c, c, \vec a_\omega\}\subseteq N$, the following holds in $N$:
There is a transitive model $M$ of $\ZFminus$ such that  $M\vDash\Phi_{\text{is}}(c_{\prec})$, $\{\omega_1,\bar c,\vec a_\omega\}\subseteq M$, and 
\begin{enumerate}[(a)]
\item $M\vDash$``$\bar c$ is a minimal $k$-witness''.
\item $\vec b(\bar c)\in M$ and writing $\vec b(\bar c)=\la b_\nu : \nu<\xi\ra$, for each $\nu<\xi$ it holds that 
$M\vDash$``$\bar c_\nu \res k$ is a minimal $k$-witness'', where $\bar c_\nu=\seqdc{\decode(b_\nu)}$.
\end{enumerate} 
\end{teqpar}
Moreover, we say $\bar c \in \powerset(\omega)^{k+1}$ is a \emph{minimal $(k+1)$-witness} if and only if
$\bar c(k)$ is the $\prec$-least $k$-localizer for $\bar c\res k$.

\medskip

Finally, we say $\bar c \in \powerset(\omega)^\omega$ is a \emph{minimal local witness} if and only if
\begin{teqpar}[$(**)$]\label{e.local.witness}
for each $k \in \omega\setminus 3$, $\bar c(k+1)$ is a $k$-localizer for $\bar c \res k$
\end{teqpar}
and we say $c\in\powerset(\omega)$  is a \emph{minimal local witness} if and only if $\seqdc{c}$ is a minimal local witness.

Given arbitrary $\bar c \in \powerset(\omega)^{\leq\omega}$ let us say $\bar c$ \emph{codes} $\vec b=\vec b(\bar c)$ if $\bar c(2)$ almost disjointly via $\mathcal F$ codes the sequence $\vec b$---just as in (b) of \ref{e.0-order}. 
\end{definition}
 
Just as before for $k=3$ we have the following crucial lemma:
\begin{lemma}\label{l.local-global.k}
Suppose $k\in\omega\setminus3$ and $\bar c \in \powerset(\omega)^k$.
There exists a $k$-localizer for $\bar c$ if and only if $\bar c$ is a minimal $k$-witness.
\end{lemma}
\begin{proof}
This is shown precisely as Lemmas~\ref{l.local-global} and \ref{l.global-local} above.
\end{proof}

We need one last observation.
\begin{lemma}\label{l.unique.local}
For each sequence $\vec b=\la b_\xi : \xi < \nu\ra$, there is at most one minimal local witness $\bar c \in \powerset(\omega)^\omega$ coding $\vec b$. Likewise, if two sequences $\bar c$ and $\bar c'$ are minimal local witnesses and $\bar c(2)=\bar c'(2)$, then $\bar c=\bar c'$. 
\end{lemma}
\begin{proof}
Suppose $\bar c$ and $\bar c'$ are minimal local witnesses coding $\vec b$.
Since $\bar c(3)$ is a $3$-localizer to $\bar c\res 3$, by Lemma~\ref{l.local-global} the latter is a minimal witness to $\vec b$. The same holds for $\bar c'$. But obviously, there can only be one minimal witness to $\vec b$, so $\bar c\res 3 = \bar c'\res 3$.
But since $\bar c(4)$ is a $4$-localizer for $\bar c\res 4$, $\bar c(3)$ is the $\prec$-least $3$-localizer by Lemma~\ref{l.local-global.k}. 
Since the same holds for $\bar c'(4)$ we have $\bar c(3)=\bar c'(3)$. Continue by induction to obtain $\bar c = \bar c'$.
The second statement follows, since if $\bar c(2)=\bar c'(2)$, also $\vec b(\bar c)=\vec b(\bar c')$. 
\end{proof}

We are now ready to begin the proof.

\begin{proof}[Proof of Theorems~\ref{t.main} and~\ref{t.better}] 
As we have stated earlier, we shall inductively construct a sequence $\la a_\nu : \nu < \omega_2\ra$ such that $\mathcal A = \{a_\nu : \nu < \omega_2\}$ will be a $\mathbf{\Pi}^1_2$ MAD family.
For the first $\omega$ elements of $\la a_\nu : \nu < \omega_2\ra$ take the sequence $\vec a_\omega = \la a_k : k\in\omega\ra$ fixed in~\ref{n.a} (since our coding functions $\code$ and $\decode$ use $\vec a_\omega$).
Fix $c_{\mathcal A}\in\powerset(\omega)$ from which both $\vec a_\omega$ and $c_{\prec}$ are computable; in the end $\mathcal A$ will be $\Pi^1_2(c_{\mathcal A})$.

Suppose we have already constructed $\la a_\nu : \nu < \xi\ra$ (where $\omega \leq \xi < \omega_1$) and assume as induction hypothesis that
for each $\nu<\xi$, letting $c_\nu=\decode(a_\nu)$ and $\bar c_\nu=\seqdc{c_\nu}$ we have that 
$a_\nu=\code(\bar c_\nu(0),\bar c_\nu(1), c_\nu)$ and $\bar c_\nu$ (or equivalently, $c_\nu$) is a minimal local witness. 
Also, let us write $d_\nu = \bar c_\nu(0)$.

Write $\mathcal A_\xi =\{ a_\nu : \nu < \xi\}$.
We will now define $a_\xi$. 
First find $d_\xi$ such that 
\begin{eqpar}\label{e.least}
$d_\xi$ is the $\prec$-least element of $[\omega]^\omega$ which is almost disjoint from every element of $\mathcal A_\xi$.
\end{eqpar}
Such $d_\xi$ exists since $\BPFA$ implies that  there is no MAD family of size less than $\omega_2$.

We now find a minimal local witness $\bar c_\xi\in \powerset(\omega)^\omega$ to $\la a_\nu : \nu < \xi\ra$. 
We shall define $\bar c_\xi \res k$ by recursion on $k>0$. 
\begin{itemize}
\item Of course, we let $\bar c_\xi(0)=d_\xi$.
\item By Fact~\ref{f.adc} there exists $c\in[\omega]^\omega$ such that $\{\nu < \xi : \lvert c \cap a_\nu \neq \omega\} = \xi\setminus\omega$. We let $\bar c_\xi(1)$ be the $\prec$-least such $c$. 
\item Also by Fact~\ref{f.adc}, there exists a subset of $\omega$ which almost disjointly via $\mathcal F$ codes $\la a_\nu : \nu < \xi\ra$;
let $\bar c(2)$ be the $\prec$-least such subset, noting that this makes $\bar c \res 3$ a minimal witness.
\end{itemize}
Just as in the proof of Lemma~\ref{l.witnessing}, there is $c\in\powerset(\omega)$ such that
\begin{eqpar}\label{e.codes}
In any suitable model $N$ such that $\{c,p_{\mathcal A}\} \subseteq N$, the following holds: Via $\mathcal F$, $c$ almost disjointly codes a well-founded model $M^0_\xi$ such that $\mathcal A_\xi\cup\{c_{\mathcal A},a^0_\xi\} \subseteq M^0$, $c \notin M^0$, $M^0 \vDash \Phi(r_\prec)$,  
and the $\Sigma_1$ formula expressing \eqref{e.least} holds in $M^0$.
\end{eqpar}
In other words, there exists a witness $c$ for $\bar c_\xi\res 3$. Let $\bar c_\xi(3)$ be the $\prec$-least such witness.
Continue recursively in the same manner for $k> 3$:
Suppose that $\bar c_\xi \res k$ is a  minimal witness.
Find a transitive model $M^k_\xi$ such that $M^{k-1}_\xi\cup\{\bar c_\xi \res k\} \subseteq M^k_\xi$, $M^k_\xi \vDash \Phi(r_\prec)$,  
and the $\Sigma_1$ formula expressing that $\bar c_\xi \res k$ is a minimal witness holds in $M^k_\xi$.
Then as above use Lemma~\ref{l.witnessing} to find a witness for $\bar c_\xi\res k$, and let $\bar c_\xi(k)$ be the least such witness. This finishes the recursive construction of $\bar c_\xi$.

Finally, we write $c_\xi$ for the element of $\powerset(\omega)$ such that $\seqdc{c_\xi}=\bar c_\xi$ and 
define
\[
a_\xi = \code(\bar c_\xi(0),\bar c_\xi(1),c_\xi), 
\]
finishing the recursive definition of $\la a_\xi : \xi < \omega_2\ra$.
Write $\mathcal A = \{a_\xi : \xi < \omega_2\}$. Clearly, by choice of $d_\xi$, $a_\xi$ and by the properties of the function $\code$ from Fact~\ref{f.coding.miller}, this is an almost disjoint family. 

\medskip

It is not hard see that $\mathcal A$ is maximal. We first point out the following simple observation:
\begin{claim}\label{c.prec}
Whenever $\nu < \xi < \omega_2$, $d_\nu \prec d_\xi$. 
\end{claim}
\begin{proof}
This is clear by the definition: Suppose otherwise that $d_\xi \preceq d_\nu$. Since $\mathcal A_\nu \subseteq \mathcal A_\xi$, $d_\xi$ is almost disjoint from every set in $\mathcal A_\nu $. So by minimality of $d_\nu$, we infer $d_\nu=d_\xi$. 
But then since $d_\nu \cap a_\nu$ is infinite, $d_\xi$ is not almost disjoint from every element of $\mathcal A_\xi$, contradicting how $d_\xi$ was chosen.
\renewcommand{\qedsymbol}{{\tiny Claim \ref{c.prec}.} $\Box$}
\end{proof} 

 \begin{claim}\label{c.mad}
 The set $\mathcal A$ is a maximal almost disjoint family.
 \end{claim}
 \begin{proof}
Suppose towards a contradiction that $d \in [\omega]^\omega\setminus\mathcal A$ and $\mathcal A\cup\{d\}$ is an almost disjoint family.
Let $\xi < \omega_2$ be the least ordinal such that $d \preceq d_\xi$; such an ordinal exists since $\prec$ well-orders the reals in ordertype $\omega_2$ and so the sequece $\la d_\xi : \xi < \omega_2\ra$ is $\prec$-cofinal in $\powerset(\omega)$. 
But since at stage $\xi$ in the construction of $\mathcal A$, $d_\xi$ was chosen to be the least element almost disjoint from every element of $\{a_\nu : \nu < \xi\}$, we have $d=d_\xi$.
Then since $a_\xi = \code(\bar c_\xi)$ and $\bar c_\xi(0)=d$,  $a_\xi \cap d$ is infinite by the properties of the function $\code$ from Fact~\ref{f.coding.miller}, contradiction.
\renewcommand{\qedsymbol}{{\tiny Claim \ref{c.mad}.} $\Box$}
\end{proof} 

We now show that $\mathcal A$ is $\Pi^1_2(p_{\mathcal A})$.
We first show:
 \begin{claim}\label{c.local.witness.pi}
 There is a $\Pi^1_2(p_{\mathcal A})$ formula $\Theta(x)$ such that $\Theta(\bar c)$ holds if and only if
 $\bar c$ is a local witness.
 \end{claim}
 \begin{proof}
Observe that for each $k\in \omega\setminus 3$ the set 
\[
\{(c, c') \in \powerset(\omega)\times\powerset(\omega)^{k} : \text{$c$ is a $k$-localizer for $c'$}\}
\]
is definable by a $\Pi^1_2(p_{\mathcal A})$ formula $\Theta_k(x,y)$.
Since $\la \Theta_k(x,y) : k \in \omega\ra$  is a recursive sequence of formulas, using a universal definable $\Pi^1_2$ truth predicate
we can find a $\Pi^1_2(p_{\mathcal A})$ formula equivalent to 
\[
(\forall k \in\omega) \;\Theta_k(\bar c(k+2),\bar c\res(k+2)).\qedhere
\]
\renewcommand{\qedsymbol}{{\tiny Claim \ref{c.local.witness.pi}.} $\Box$}
\end{proof} 
 Let now $\Psi(a)$ be defined as follows:
 \[
 \Psi(a) \stackrel{\text{def}}{\iff} 
 (\forall \bar c \in (\powerset(\omega)^\omega)\; \bar c = \decode(a) \Rightarrow
 \big( a = \code(\bar c) \land \Theta(\bar c)\big). 
 \]
 Clearly this formula is $\Pi^1_2(p_{\mathcal A})$.
We will show that $\Psi(a) \iff a\in \mathcal A$. The non-trivial direction is ``$\Rightarrow$,'' which we show first.
 \begin{lemma}\label{l.define.A}
$(\forall a\in [\omega]^\omega) \; \Psi(a)\Rightarrow a \in \mathcal A$.
 \end{lemma}
 \begin{proof}
For each $a$ such that $\Psi(a)$, we know that $\bar c = \decode(a)$ is a local witness, and so 
$\vec b(\bar c)$ is defined, namely as the unique sequence coded by $\bar c(1)$ as in \ref{e.0-order} (b).
For each such $a$ let us write $\vec b(a)$ for this sequence.
 We need the following claim:
  \begin{claim}\label{c.init}
For any $a \in [\omega]^\omega$ such that $\Psi(a)$ holds, $\vec b(a)$ 
is an initial segment of $\la a_\nu : \nu < \omega_2\ra$.
 \end{claim}
 \begin{proof}
Towards a contradiction, let
$a \in \powerset(\omega)$ be $\prec$-least such that $\Psi(a)$ but $\vec b(a)$ 
is \emph{not} an initial segment of $\la a_\nu : \nu < \omega_2\ra$. 

Write $\bar c=\decode(a)$ and suppose $\vec b(a)=\la b_\nu : \nu<\xi\ra$. 
Let $\nu < \xi$ be least such that $b_\nu \neq a_\nu$.
Recall that by Lemma~\ref{l.local-global}, $\bar c(k)$ is a $k$-localizer for $c\res (k)$ for each $k\in\omega\setminus 3$.

In particular $\bar c(3)$ is a $3$-localizer for $c\res 3$.
Then by (b) in \ref{e.1-order} and by the proof of Lemma~\ref{l.local-global}, letting $\bar c^*_\nu=\decode(b_\nu)$ it holds that $\bar c^*_\nu \res 3$ is a minimal witness.
More generally, since $\bar c(k)$ is a $k$-localizer for $c\res (k+2)$, by (b) in \ref{e.k+1-order}
we see that $\bar c^*_\nu \res k$ is a minimal witness.
Since this holds for each $k\in\omega$, $\bar c^*_\nu$ is a local witness; and since $b_\nu = \code(\bar c^*_\nu)$ we infer $\Psi(b_\nu)$ holds.
 \begin{subclaim*}
It holds that $b_\nu \prec a$.
 \end{subclaim*}
 \begin{proof}
By (c) of \ref{e.1-order} there is a model $M$ of $\ZFminus$ with $\{b_\nu,\bar c\res 3\}\subseteq M$ such that $M\cap\powerset(\omega)$ is an initial segment of $\prec$, and moreover $\bar c(3) \notin M$. 
So $a \preceq b_\nu$ leads to a contradiction, since then $\bar c= \decode(a) \in M$ and so $\bar c(3)\in M$. 
\renewcommand{\qedsymbol}{{\tiny Subclaim.} $\Box$}
\end{proof} 
Thus by minimality of $a$, $\vec b(b_\nu)$ is an initial segment of $\la a_\xi : \xi < \omega_2\ra$.
In fact, since $\bar c \res 3$ is a $3$-witness, (b) in \ref{e.0-order} tells us that $\vec b (b_\nu)$ is a sequence of length $\nu$.
We conclude $\vec b (b_\nu) = \la a_\xi : \xi < \nu\ra$.
But then since $\bar c^*_\nu$ and $\bar c_\nu$ are both local witnesses for the sequence $\la a_\xi : \xi < \nu\ra$, we must have $\bar c^*_\nu=\bar c_\nu$ by Lemma~\ref{l.unique.local}.
If follows that $a_\nu = \code(\bar c_\nu)= \code (\bar c^*_\nu) = b_\nu$, contradiction.
\renewcommand{\qedsymbol}{{\tiny Claim \ref{c.local.witness.pi}.} $\Box$}
\end{proof} 
By the claim, if $\Psi(a)$ holds, we can fix $\nu<\omega_2$ such that $\vec b(a) = \la a_\xi : \xi < \nu\ra$.
Then by uniqueness of the minimal local witnesses (again Lemma~\ref{l.unique.local}) for $\la a_\xi : \xi < \nu\ra$, $\decode(a)=\bar c=\bar c_\nu$
and $a=\code(\bar c_\nu)=a_\nu$.
\renewcommand{\qedsymbol}{{\tiny Lemma \ref{l.define.A}.} $\Box$}
\end{proof} 
Finally, it is clear by construction that for any $\xi < \omega_2$, $a_\xi = \code(\bar c_\xi)$ and $\bar c_\xi$ is a local witness.
 Therefore $\Psi(a_\xi)$ holds. So $a\in \mathcal A\Rightarrow \Psi(a)$. 
\renewcommand{\qedsymbol}{{\tiny Theorems~\ref{t.main}\&~\ref{t.better}.} $\Box$}
\end{proof}

\section{Questions}

Finally, the proof we give is obviously more widely applicable, e.g., for other infinite combinatorial objects such as maximal independent families.
The authors intend to find an axiomatization of the objects to which it will be applicable, perhaps in the style of Z.\ Vidnyanski's axiomatization of Miller's procedure.

In \cite{004S5}, Schindler showed the consistency of $\BPFA +$``Every projective set of reals is Lebesgue-measurable'', assuming that there is a $\Sigma_2$-correct regular cardinal above a remarkable cardinal. In light of this and Corollary~\ref{c.main} the following is a natural question:

\begin{question}
Can one prove the consistency of $\BPFA + \text{``All } \mathbf{\Pi^1_2} \text{ mad families are finite."}$ from a $\Sigma_2$-correct regular cardinal above a remarkable cardinal? 
\end{question}

\begin{question}
Can $\BPFA$ be replaced by the Bounded Forcing Axiom for Axiom A in Theorem~\ref{t.main}?
\end{question}

\begin{question}
Can the anti-large cardinal assumption be weakened? 
Can we assume a forcing axiom \emph{stronger} than $\BPFA$, but still compatible with such an assumption, and derive a form of Theorem~\ref{t.main}?
\end{question}

\bibliography{definable-mad-families-BPFA}{}
\bibliographystyle{amsplain}

\end{document}